\documentclass[pdflatex,sn-mathphys-num]{sn-jnl}

\usepackage{graphicx}%
\usepackage{multirow}%
\usepackage{amsmath,amssymb,amsfonts}%
\usepackage{amsthm}%
\usepackage{mathrsfs}%
\usepackage[title]{appendix}%
\usepackage{xcolor}%
\usepackage{hyperref}
\usepackage{cleveref} 
\usepackage{textcomp}%
\usepackage{manyfoot}%
\usepackage{booktabs}%
\usepackage{algorithm}%
\usepackage{algorithmicx}%
\usepackage{algpseudocode}%
\usepackage{listings}%

\theoremstyle{thmstyleone}%
\newtheorem{theorem}{Theorem}
\newtheorem{proposition}[theorem]{Proposition}%
\newtheorem{lemma}[theorem]{Lemma}%

\theoremstyle{thmstyletwo}%
\newtheorem{remark}{Remark}%

\theoremstyle{thmstylethree}%
\newtheorem{definition}{Definition}%

\begin{document}

\title[On the Ricci symmetries of a K\"ahler manifold]{On the Ricci symmetries of a K\"ahler manifold}


\author*[1,2]{\fnm{Jorge} \sur{Alc\'azar}}\email{f82algoj@uco.es}

\affil*[1]{\orgdiv{Department of Mathematics}, \orgname{University of C\'ordoba}, \orgaddress{\street{Campus de Rabanales}, \city{C\'ordoba}, \postcode{14071}, \country{Spain}}}


\abstract{The main purpose of the present paper is to investigate the symmetry properties of a K\"ahler manifold involving the Ricci tensor. In this context, the most symmetric manifolds are K\"ahler--Einstein spaces, and their natural generalizations are Ricci-parallel K\"ahler manifolds, Ricci-semisymmetric K\"ahler manifolds and holomorphically Ricci-pseudosymmetric K\"ahler manifolds. Unlike their Riemannian counterparts, we prove that all these conditions also admit a characterization solely in terms of holomorphic planes, analogously to the symmetries related to the Riemannian curvature tensor in K\"ahler manifolds. A key finding is that the concept of holomorphic Ricci pseudosymmetry is distinct from the classical Ricci-pseudosymmetric condition introduced by Deszcz. By carefully analyzing the interplay between these definitions, we clarify the precise geometric role of the so-called Ricci curvature of Deszcz. Additionally, we also present a geometric interpretation of the complex Tachibana--Ricci tensor and we establish a new criterion for a K\"ahler manifold to be Einstein based on holomorphic planes.}





\maketitle

\section{Introduction}\label{sec1}

Let $(M,g)$ be an $n$-dimensional Riemannian manifold with metric $g$. We denote by $\nabla$ the Levi--Civita connection and by $R$ its associated $(1,3)$-Riemann--Christoffel curvature tensor. The endomorphisms $X\wedge_{g} Y$ and $R(X,Y)$ of the Lie algebra of vector fields $\mathfrak{X}(M)$ on $M$ are defined by
\begin{equation}
	(X\wedge_{g} Y)Z=g(Y,Z)X-g(X,Z)Y,
\end{equation}
and
\begin{equation}
	R(X,Y)Z=\nabla_{X} \nabla_{Y}Z-\nabla_{Y} \nabla_{X}Z-\nabla_{[X,Y]}Z,
\end{equation}
for any vector fields $X,Y,Z \in \mathfrak{X}(M)$. 

At a point $p\in M$, let $\pi=x\wedge y$ be a tangent plane to $M$ at $p$, spanned by two vectors $x$ and $y$. Then, the real number
\begin{equation}
	K(p,\pi)=\dfrac{g(R(x,y)y,x)}{g((x\wedge_{g}y)y,x)},
\end{equation}
depends only on the point $p$ and the plane $\pi$, and is called the sectional curvature of $M$ at $p$ for the plane section $\pi$.

Given an orthonormal basis $\{E_{1},\ldots, E_{n}\}$, the Ricci tensor on a Riemannian manifold is given by
\begin{equation*}
	S(X,Y)=\sum_{i=1}^{n}R(E_{i},X,Y,E_{i}),
\end{equation*}
for any $X,Y\in \mathfrak{X}(M)$, and the scalar curvature is defined as the trace of the Ricci tensor
\begin{equation}
	\text{Scal}=\text{tr}_{g}S.
\end{equation}

At a point $p\in M$, let $\{e_{1},e_{2},\ldots,e_{n}\}$ be an orthonormal basis of $T_{p}M$. The Ricci curvature in the direction of $e_{1}$, $\operatorname{Ric}(e_{1})$, can be written as
\begin{equation}
	\operatorname{Ric}(e_{1})=\sum_{j=2}^{n}K(p,e_{1}\wedge e_{j}).
\end{equation}

A Riemannian manifold $(M,g)$ is said to be \textit{locally Ricci flat} when its Ricci tensor vanishes. The simplest non-Ricci flat Riemannian manifolds are the \textit{spaces of constant Ricci curvatures}, also known as \textit{Einstein spaces} \cite{JHSV}, for which the Ricci tensor takes the form:
\begin{equation}
	S=\lambda g.
\end{equation}

Note that every Ricci flat manifold is Einstein, however, the converse is not true. From the works of Cartan \cite{Cartan}, as a generalization of these, \textit{Ricci-parallel spaces} appear as those such that $\nabla S=0$, characterized by having constant scalar curvature. 

Going one step further we can observe that the integrability condition of $\nabla S =0$ is $R\cdot S=0$. Therefore, every Ricci-parallel space also satisfies $R\cdot S=0$. This fact allows us to define \textit{Ricci-semisymmetric spaces} as those Riemannian manifolds satisfying $R\cdot S=0$. The class of Ricci-semisymmetric spaces includes the class of Ricci-parallel spaces as a proper subset. However, in general the converse is not true (see, for example, \cite{AD}). We recall at this point that $R\cdot S$ is the $(0,4)$-tensor on $M$ given by
\begin{equation}
	\begin{split}
		&(R\cdot S)(X_{1},X_{2};X,Y)=(R(X,Y)\cdot S)(X_{1},X_{2})\\
		&=-S(R(X,Y)X_{1},X_{2})-S(X_{1},R(X,Y)X_{2}),
	\end{split}
\end{equation}
for any vector fields $X_{1},X_{2},X,Y \in \mathfrak{X}(M)$. 

As was proved in \cite[Corollary 3]{JHSV} Ricci-semisymmetric manifolds are characterized as those manifolds whose Ricci curvature function is invariant, up to second order terms, under parallel transport of any vector $v$ at any point $p$ in $M$ around any infinitesimal coordinate parallelogram $\mathcal{P}$ cornered at $p$. It is immediate to observe that any Riemannian surface satisfies $R\cdot S=0$, so the concept of Ricci semisymmetry becomes relevant in the case of Riemannian manifolds $(M^{n},g)$ of dimension $n\geq 3$. Note that every semisymmetric manifold is Ricci-semisymmetric. The converse, however, is not true in general \cite{JHSV}.

As a natural generalization of these latter spaces, \textit{Ricci-pseudosymmetric spaces in the sense of Deszcz} \cite{Olszak} are defined as Riemannian manifolds $(M^{n},g)$ of dimension $n\geq 3$ for which
\begin{equation}
	R\cdot S=L_{S} \hspace{0.1cm}Q(g,S),
\end{equation}
where $L_{S}\in \mathcal{C}^{\infty}(M)$ and $Q(g,S)$ is the \textit{Tachibana--Ricci tensor} defined by
\begin{equation}
	\begin{split}
		Q(g,S)(X_{1},X_{2};X,Y)=&((X\wedge_{g}Y)\cdot S)(X_{1},X_{2})\\
		=&-S((X\wedge_{g}Y)X_{1},X_{2})-S(X_{1},(X\wedge_{g}Y)X_{2}),
	\end{split}
\end{equation}
for any vector fields $X_{1},X_{2},X,Y \in \mathfrak{X}(M)$. The Tachibana--Ricci tensor is the simplest $(0,4)$-tensor with the same algebraic symmetries as $R\cdot S$.

\begin{table}[h]
	\caption{Natural Ricci symmetries on a Riemannian manifold.}
	\label{tab:ricci-symmetries}
	\begin{tabular}{ccc}
		\toprule
		Ricci flat & $\qquad$ & $S=0$\\
		\midrule
		Einstein & & $S=\lambda g$\\
		\midrule
		Ricci-parallel & & $\nabla S =0$\\
		\midrule
		Ricci-semisymmetric & & $R\cdot S =0$\\
		\midrule
		Deszcz Ricci-pseudosymmetric & & $R\cdot S = L \hspace{0.1cm}Q(g,S)$\\
		\botrule
	\end{tabular}
\end{table}

A Riemannian manifold is Einstein if and only if its Tachibana--Ricci tensor vanishes identically, see \cite[Lemma 4]{JHSV}. Therefore, if $(M^{n},g)$, $n\geq 3$, is not Einstein, the set of points where the Tachibana--Ricci tensor does not vanish identically is a non-empty open subset $\mathcal{U}\subseteq M$. Given a point $p\in \mathcal{U}$, a direction $d$, generated by a vector $v\in T_{p}M$, is said to be curvature dependent on a plane $\bar{\pi}=x\wedge y\subset T_{p}M$ if $Q(g,S)(v,v;x,y)\neq 0$. In this context, given $p\in \mathcal{U}$ and one direction $d$, generated by a vector $v\in T_{p}M$, curvature dependent on a plane $\bar{\pi}=x\wedge y\subset T_{p}M$, the \textit{Ricci curvature of Deszcz} $L_{S}(p,d,\bar{\pi})$ of the direction $d$ and the plane $\bar{\pi}$ is defined as the scalar
\begin{equation} \label{RCD}
	L_{S}(p,d,\bar{\pi})=\dfrac{(R\cdot S)(v,v;x,y)}{Q(g,S)(v,v;x,y)}.
\end{equation}
The above definitions are independent of the choice of bases for the direction $d$ and the plane $\bar{\pi}$.

Given $p\in M$ and a direction $d$, generated by a vector $v\in T_{p}M$, curvature dependent on a plane $\bar{\pi}=x\wedge y$, $Q(g,S)(v,v;x,y)$ measures the change of the Ricci curvature, $\operatorname{Ric}(v)$, under an operation involving infinitesimal rotations around a point $p$, without leaving this point. On the other hand, $(R\cdot S)(v,v;x,y)$ measures the change of the Ricci curvature, $\operatorname{Ric}(v)$, after parallel transport of the vector $v$ around any infinitesimal coordinate parallelogram $\mathcal{P}$ cornered at any point $p\in M$ \cite{JHSV}.

It is obvious that if $(M,g)$ is a Ricci-pseudosymmetric manifold in the sense of Deszcz, all its Ricci curvatures of Deszcz are independent of the directions and planes. Moreover, this is a sufficient condition: a Riemannian manifold $M$ of dimension $n\geq 3$ is Ricci-pseudosymmetric in the sense of Deszcz if and only if, at each point $p\in \mathcal{U}$, for every direction $d$, given by a vector $v$, and every plane $\pi$ in $T_{p}M$, such that the direction $d$ is curvature dependent on the plane $\pi$, $L_{S}(p,d,\pi)=L_{S}(p)$ for some function $L_{S}\in \mathcal{C}^{\infty}$ \cite[Theorem 10]{JHSV}. Note that while every pseudosymmetric manifold in the sense of Deszcz is Ricci-pseudosymmetric, the converse is not true \cite{JHSV}.

Throughout this work we will consider K\"ahler manifolds $(M^{2n},g,J)$ satisfying different symmetries. Specifically, we will study the symmetries related to Einstein conditions, Ricci-parallel, Ricci-semisymmetric, and holomorphically Ricci-pseudosymmetric. It should be noted in this regard that the tensors $R\cdot S$ and $Q(g,S)$ do not present the same symmetries and properties when considering the complex structure $J$. For this reason, in 1989 Olszak proposed an alternative concept of pseudosymmetry for K\"ahler manifolds \cite{OlszakB}. We will refer to this concept of Ricci pseudosymmetry as \textit{holomorphic Ricci pseudosymmetry}.

The present paper is organized as follows. In Section~\ref{sec:setup}, some basic notions for K\"ahler manifolds will be recalled and two new algebraic results concerning the symmetries of $(0,4)$-tensor will be presented: Lemma~\ref{lem:simRSJ} and Proposition~\ref{prop:auxalg}. In Section~\ref{sec:complexRTach}, the complex Tachibana--Ricci tensor is defined and, analogous to the Riemannian case, a geometric interpretation will be given. This tensor allows us to give a characterization result for K\"ahler--Einstein manifolds in terms of holomorphic planes: Theorem~\ref{thm:Einsteinchac}. In Section~\ref{sec:Ricciparsem}, characterization results for Ricci-parallel and Ricci-semisymmetric K\"ahler manifolds are provided: Theorem~\ref{thm:Riccipar} and Theorem~\ref{thm:Riccisem}. Finally, Section~\ref{sec:HRPS} is dedicated to holomorphically Ricci-pseudosymmetric K\"ahler manifolds. The relationships between holomorphic Ricci pseudosymmetry, Ricci pseudosymmetry in the sense of Deszcz and Ricci curvature of Deszcz are studied. In particular, a characterization result for holomorphically Ricci-pseudosymmetric K\"ahler manifolds is provided in terms of Ricci curvatures: Theorem~\ref{thm:HRpseud}.

\section{Set up} \label{sec:setup}

A K\"ahler manifold is a complex manifold $(M^{2n},g,J)$ of real dimension $2n$, where $g$ is a Hermitian metric on $M$ and the complex structure $J$ is parallel; that is, $\nabla J=0$. Given a K\"ahler manifold $(M,g,J)$, it is possible to define the complex metric endomorphism by
\begin{equation} \label{EMC}
	(X\wedge^{c}_{g}Y)Z=(X\wedge_{g}Y)Z+(JX\wedge_{g}JY)Z-2g(JX,Y)JZ,
\end{equation}
for any $X,Y,Z \in \mathfrak{X}(M)$. The simplest non-Ricci flat K\"ahler manifolds are K\"ahler--Einstein manifolds, for which the Ricci tensor takes the form
\begin{equation}
	S(X,Y)=\dfrac{\widetilde{c}}{2}(n+1)g(X,Y),
\end{equation}
for any $X,Y \in \mathfrak{X}(M)$, with $n$ the complex dimension of the manifold (see, for instance, \cite[Theorem 7.5, Chapter IX]{KN_1969}).
\begin{remark}
	Note that every complex space form is a K\"ahler--Einstein manifold, but the converse is not true. See, for instance, $K3$ surfaces \cite{Siu} endowed with a Ricci-flat K\"ahler metric. In this case, $\lambda=0$, and consequently $\widetilde{c}=0$. A K\"ahler manifold with constant holomorphic sectional curvature $\widetilde{c}=0$ is flat, i.e, locally isometric to $\mathbb{C}^{2}$ with the Euclidean metric. A compact flat K\"ahler manifold, $X$, is finitely covered by a complex torus $\mathbb{C}^{2}/ \varLambda$, for which the second Chern class is $c_{2}(X)=0$ \cite[Corollary 11.27]{Besse}. However, a $K3$ surface, $Y$, cannot be flat, because $c_{2}(Y)=24\neq 0$ \cite{BV}.
\end{remark}

We begin by recalling some basic properties of the Ricci tensor involving the complex structure $J$ on a K\"ahler manifold \cite{Olszak}.
\begin{lemma} \label{lem:simSJ}
	The Ricci tensor of a K\"ahler manifold satisfies the following properties
	\begin{itemize}
		\item [a)] $S(X,Y)=S(Y,X)$,
		\item [b)] $S(X,JX)=0$,
		\item [c)] $S(JX,JY)=S(X,Y)$,
		\item [d)] $S(X,JY)=-S(JX,Y)$,
	\end{itemize}
	for any $X,Y \in \mathfrak{X}(M)$.
\end{lemma}

Moreover, the Ricci 2-form $\rho$, $\rho(X,Y)=S(X,JY)$, is closed, and consequently
\begin{equation*}
	(\nabla_{X}S)(Y,JZ)+(\nabla_{Y}S)(Z,JX)+(\nabla_{Z}S)(X,JY)=0,
\end{equation*}
for any $X,Y,Z \in \mathfrak{X}(M)$ \cite{Olszak}.

When studying further the symmetries of any K\"ahler manifold, or in general of any Riemannian manifold, the tensor $R\cdot S$ is essential. In addition to the symmetries inherent to the Riemannian case (see, for example, \cite{JHSV}), in the case of a K\"ahler manifold it behaves well with respect to $J$. Taking into account the behavior of $S$ with respect to $J$, we have the following lemma.
\begin{lemma} \label{lem:simRSJ}
	The tensor $R\cdot S$ of a K\"ahler manifold $(M,g,J)$ satisfies
	\begin{equation*}
		(R\cdot S)(JX_{1},JX_{2};X,Y)=(R\cdot S)(X_{1},X_{2};JX,JY)=(R\cdot S)(X_{1},X_{2};X,Y).
	\end{equation*}
\end{lemma}
\begin{proof}
	The proof follows immediately from the definition of $R\cdot S$ and taking into account \cite[Prop. 4.5, Chap. IX]{KN_1969} and Lemma \ref{lem:simSJ}.
\end{proof}

Thanks to the following result, $(0,4)$-tensors on a K\"ahler manifold satisfying the same properties as $R\cdot S$ are characterized by their behavior when applied to holomorphic planes:

\begin{proposition} \label{prop:auxalg}
	Let $V^{2n}$ be a real vectorial space endowed with a complex structure $J$ and let $R$ and $T$ be two $(0,4)$-tensors on $V^{2n}$ satisfying the following properties:
	\begin{itemize}
		\item [a)] $R(x_{1},x_{2},x_{3},x_{4})=-R(x_{1},x_{2},x_{4},x_{3})$,
		\item [b)] $R(x_{1},x_{2},x_{3},x_{4})=R(x_{2},x_{1},x_{3},x_{4})$,
		\item [c)] $R(x_{1},x_{2},Jx_{3},Jx_{4})=R(Jx_{1},Jx_{2},x_{3},x_{4})=R(x_{1},x_{2},x_{3},x_{4})$,
		\item [d)] $R(x_{1},Jx_{2},x_{3},x_{4})=-R(Jx_{1},x_{2},x_{3},x_{4})$,
		\item [e)] $R(x_{1},x_{2},x_{3},Jx_{4})=-R(x_{1},x_{2},Jx_{3},x_{4})$,
	\end{itemize}
	for any $x_{1},x_{2},x_{3},x_{4} \in V $. If for every $u,v \in V$ we have
	\begin{equation*}
		R(u,u,x,Jx)=T(u,u,x,Jx),
	\end{equation*}
	then $R=T$.
\end{proposition}
\begin{proof}
	We define $W:=R-T$, which has the same symmetries as $R$ and $T$. By hypothesis $W(u,u,x,Jx)=0$, for all $u,x\in V^{2n}$.
	
	For any $u,v,x\in V^{2n}$, we have
	\begin{equation*}
		0=W(u+v,u+v,x,Jx)=W(u,u,x,Jx)+W(v,v,x,Jx)+2W(u,v,x,Jx),
	\end{equation*}
	hence $W(u,v,x,Jx)=0$.
	
	On the other hand, for any $u,v,x,y\in V^{2n}$, we have
	\begin{equation*}
		\begin{split}
			0 &= W(u,v,x+y,Jx+Jy) = W(u,v,x,Jx)+W(u,v,y,Jy)\\
			&\quad +W(u,v,x,Jy)+W(u,v,y,Jx)= 2W(u,v,x,Jy),
		\end{split}
	\end{equation*}
	and consequently, $W(u,v,x,Jy)=0$. Interchanging $y$ for $Jy$ we obtain
	\begin{equation*}
		W(u,v,x,y)=0,
	\end{equation*}  
	as desired.
\end{proof}

\section{The complex Tachibana--Ricci tensor and K\"ahler--Einstein manifolds} \label{sec:complexRTach}

In analogy with the classical Tachibana--Ricci tensor, on a K\"ahler manifold we define the complex Tachibana--Ricci tensor as follows.

\begin{definition} Given a K\"ahler manifold $(M^{2n},g,J)$, the complex Tachibana--Ricci tensor is defined as the $(0,4)$-tensor on $M$ given by
	\begin{equation}
		\begin{split}
			&Q^{c}(g,S)(X_{1},X_{2};X,Y)=((X\wedge^{c}_{g}Y)\cdot S)(X_{1},X_{2})=\\
			&-S((X\wedge^{c}_{g}Y)X_{1},X_{2})-S(X_{1},(X\wedge^{c}_{g}Y)X_{2}),
		\end{split}
	\end{equation}
	for any vector fields $X_{1},X_{2},X,Y \in \mathfrak{X}(M)$.
\end{definition}

\begin{remark}
	The choice of notation $Q^{c}(g,\operatorname{S})$ for the complex Tachibana--Ricci tensor is due to its analogy with the classical Tachibana--Ricci tensor for a Riemannian manifold \cite{Olszak}.
\end{remark}

The complex Tachibana--Ricci tensor can act as the simplest $(0,4)$-tensor with the same algebraic symmetries and the same properties with respect to $J$ as $R\cdot S$. Moreover, taking into account (\ref{EMC}), for each $p\in M$ and any $v,w,x,y\in T_{p}M$, it is immediate to obtain the following relation between the complex Tachibana--Ricci tensor and its classical counterpart,
\begin{equation} \label{TC}
	\begin{split}
		&Q^{c}(g,S)(v,w;x,y)=-S((x\wedge^{c}_{g}y)v,w)-S(v,(x\wedge^{c}_{g}y)w)=\\
		&-S((x\wedge y)v,w)-S(v,(x\wedge y)w)-S((Jx\wedge Jy)v,w)-S(v,(Jx\wedge Jy)w)\\
		&+2g(x,Jy)S(Jv,w)+2g(x,Jy)S(v,Jw)=\\
		&Q(g,S)(v,w;x,y)+Q(g,S)(v,w;Jx,Jy).
	\end{split}
\end{equation}

Given $p\in \mathcal{U}$ and two planes in $T_{p}M$, $\pi=v\wedge w$, $\bar{\pi}=x\wedge y$, it is easy to verify that $Q(g,S)(v,w;x,y)$ and $Q^{c}(g,S)(v,w;x,y)$ do not depend on the choice of bases for $\pi$ and $\bar{\pi}$, so we can simply write $Q(g,S)(\pi;\bar{\pi})$ and $Q^{c}(g,S)(\pi;\bar{\pi})$. Taking this notation into account, (\ref{TC}) reads
\begin{equation*}
	Q^{c}(g,S)(\pi;\bar{\pi})=Q(g,S)(\pi;\bar{\pi})+Q(g,S)(\pi;J\bar{\pi}),
\end{equation*}
where $J\bar{\pi}=Jx\wedge Jy$. It should also be noted that when the second of the two planes is holomorphic, that is, invariant under the action of $J$, the Tachibana--Ricci tensor and the complex Tachibana--Ricci tensor are proportional. Indeed, it is easy to verify that
\begin{equation} 
	Q^{c}(g,S)(\pi;\bar{\pi}^{h})=2Q(g,S)(\pi;\bar{\pi}^{h}).
\end{equation}
Moreover, given a direction $d$ generated by a vector $u\in T_{p}M$ and a holomorphic plane $\bar{\pi}^{h}$, we also have the following relation between both tensors
\begin{equation} \label{RH}
	Q^{c}(g,S)(u,u;\bar{\pi}^{h})=2Q(g,S)(u,u;\bar{\pi}^{h}),
\end{equation}
where we use the superscript $^{h}$ to denote holomorphic planes.

Note that if the first of the planes is holomorphic we have
\begin{equation}
	Q^{c}(g,S)(\pi^{h};\bar{\pi})=0.
\end{equation}

As a first result in this section we provide the following characterization for K\"ahler--Einstein manifolds.

\begin{theorem} \label{thm:Einsteinchac}
	A K\"ahler manifold $(M,g,J)$ is K\"ahler--Einstein if and only if $Q^{c}(g,S)(u,u;x,Jx)=0$ for any $u,x\in T_{p}M$.
\end{theorem}

\begin{proof}
	The necessary condition follows immediately, since if $M$ is K\"ahler--Einstein, $\operatorname{Ric}=\lambda g$, so $Q^{c}(g,S)=Q^{c}(g,\lambda g)=0$.
	
	For the sufficient condition, using the properties of $S$ and (\ref{RH}) we have
	\begin{equation*}
		0=Q^{c}(g,S)(u,u;x,Jx)=-2S((x\wedge^{c}Jx)u,u),
	\end{equation*}
	and thus,
	\begin{equation*} \label{ref1}
		g(Jx,u)S(x,u)=g(x,u)S(Jx,u).
	\end{equation*}
	
	We decompose $u$ into its projection onto the plane spanned by $x$ and $Jx$ and its orthogonal projection, $u=\alpha x+\beta Jx+w$, and we obtain
	\begin{equation*}
		\beta S(x,w)=\alpha S(Jx,w).
	\end{equation*}
	
	In particular, if $\beta=0$ and $\alpha\neq 0$,
	\begin{equation*}
		S(Jx,w)=0.
	\end{equation*}
	
	Consider an orthonormal basis of $T_{p}M$, $\{e_{i},Je_{i}\}_{i=1,\ldots,n}$. Taking $x=e_{i}$, $w=Je_{j}$, with $i\neq j$, we obtain
	\begin{equation*}
		S(Je_{i},Je_{j})=S(e_{i},e_{j})=0.
	\end{equation*}
	
	And the same way, taking $x=e_{i}$, $w=e_{j}$, with $i\neq j$, we obtain
	\begin{equation*}
		S(Je_{i},e_{j})=S(e_{i},Je_{j})=0.
	\end{equation*}
	
	From the properties of $S$ we know that $S(e_{i},Je_{i})=0$. Therefore, the only non-zero terms will be the diagonal elements; that is, the terms $S(e_{i},e_{i})$ and $S(Je_{i},Je_{i})$, so that
	\begin{equation} \label{ref3}
		S=\begin{pmatrix}
			D & 0 \\
			0 & D
		\end{pmatrix},
	\end{equation}
	where $D$ is a diagonal matrix with entries $\lambda_i = S(e_i,e_i)=S(Je_{i},Je_{i})$.
	
	Now applying Proposition~\ref{prop:auxalg} to the complex Tachibana--Ricci tensor, taking into account the hypothesis of the statement, we have that $Q^{c}(g,S)\equiv 0$, so in particular
	\begin{equation*}
		\begin{split}
			0 &= Q^{c}(g,S)(e_{i},e_{j};e_{i},e_{j}) \\
			&= -S((e_{i}\wedge^{c} e_{j})e_{i},e_{j})-S(e_{i},(e_{i}\wedge^{c} e_{j})e_{j})\\
			&= +S(e_{j},e_{j})-S(e_{i},e_{i})=\lambda_{j}-\lambda_{i}.
		\end{split}
	\end{equation*}
	
	Therefore, $\lambda_{i}=\lambda_{j}$ for all $i\neq j$, $i,j=1,\ldots, n$, and from (\ref{ref3}) we obtain that $S=\lambda I$ for the given orthonormal basis. Consequently, $S=\lambda g$, as desired.
\end{proof}

\medskip

Our next goal is to give a geometric interpretation of the tensor $Q^{c}(g,S)$ in a similar way to that given in \cite{JHSV} for $Q(g,S)$ in the Riemmanian case.

As a preliminary step, we recall the geometric interpretation of the metric endomorphism applied to a vector, focusing on the case of K\"ahler manifolds. In this sense, taking into account \cite{AAC}, the vector $(x\wedge_{g} y)$ measures the first-order change of a vector $z$ after an infinitesimal rotation in the plane $\pi =x\wedge y$ at the point $p$.

So, if we take $p\in \mathcal{U}$, a direction $d$ generated by a vector $v\in T_{p}M$, and a plane $\pi =x\wedge y$, where $\{x,y\}$ are orthonormal, infinitesimally rotating the projection of $v$ onto $\pi$ we obtain
\begin{equation*}
	v'=v+\varepsilon(x\wedge_{g} y)v+\mathcal{O}(\varepsilon^{2}).
\end{equation*}

Infinitesimally rotating again the projection of this new vector onto $J\pi=Jx\wedge Jy$ we obtain
\begin{equation*}
	v''=v+\varepsilon(x\wedge_{g} y)v+\varepsilon(Jx\wedge_{g} Jy)v+\mathcal{O}(\varepsilon^{2}).
\end{equation*}

Comparing the Ricci curvatures $\operatorname{Ric}(v)$ and $\operatorname{Ric}(v'')$, we have
\begin{equation}
	\begin{split}
		\operatorname{Ric}(v'')=&\operatorname{Ric}(v)-\varepsilon Q(g,S)(v,v;x,y)-\varepsilon Q(g,S)(v,v,Jx,Jy)+\mathcal{O}(\varepsilon^{2})\\
		=&\operatorname{Ric}(v)-\varepsilon Q^{c}(g,S)(v,v;x,y)+\mathcal{O}(\varepsilon^{2}).
	\end{split}
\end{equation}

\begin{figure}
	\centering
	\includegraphics[width=10cm]{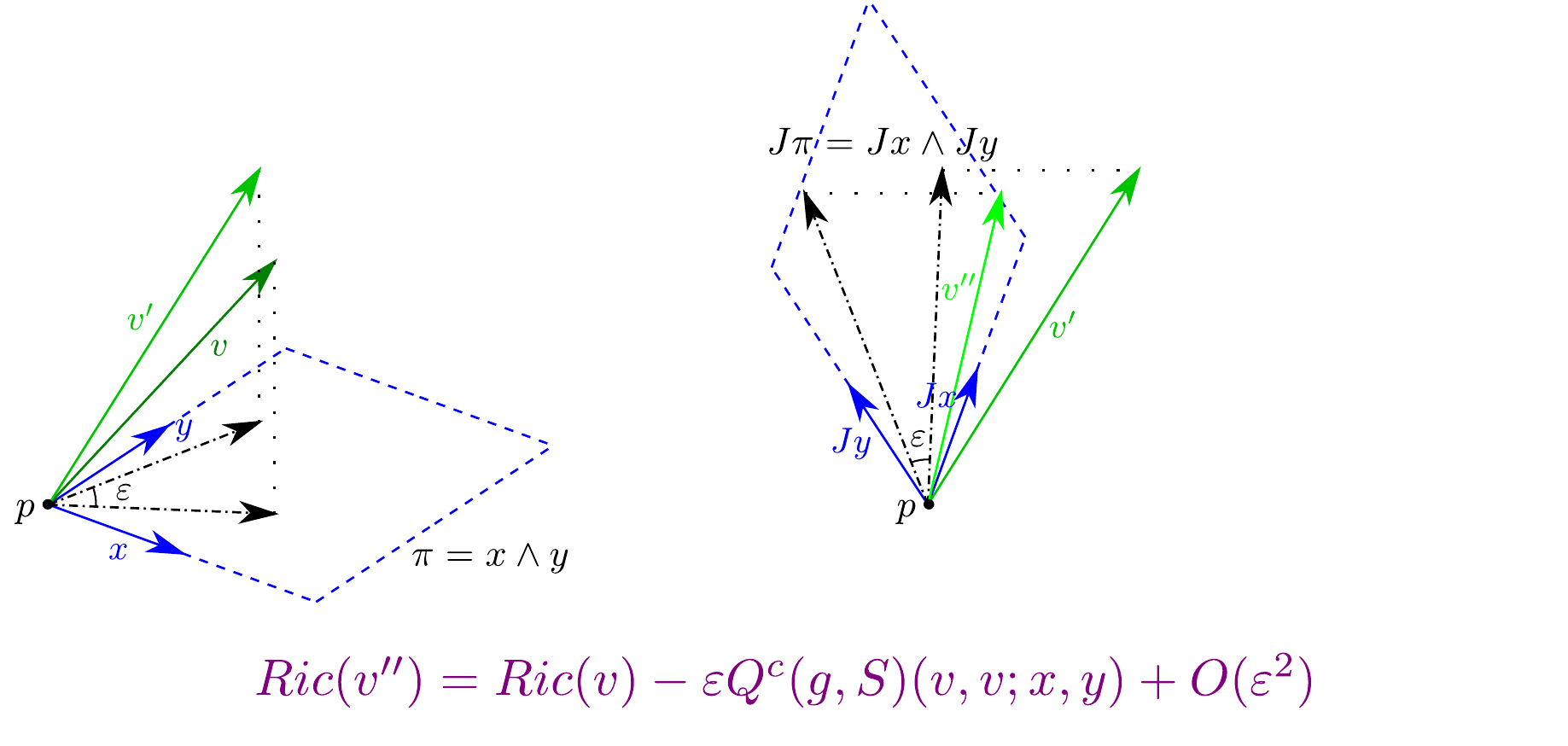}
	\caption{Geometric interpretation of the complex Tachibana--Ricci tensor}
\end{figure}
\vspace*{-0.2cm}

In conclusion, the components of $Q^{c}(g,S)(v,v;x,y)$ measure the change of the Ricci curvature, $\operatorname{Ric}(v)$, after operations involving infinitesimal rotations in $\pi$ and $J\pi$ up to second order.

\section{Ricci-paralell and Ricci-semisymmetric K\"ahler manifolds} \label{sec:Ricciparsem}

As we recalled in the introduction, Ricci-parallel spaces, that is, those for which $\nabla S=0$, are a natural generalization of K\"ahler--Einstein spaces. Our next result for K\"ahler manifolds shows that in the previous characterization we can consider only holomorphic planes.

\begin{theorem} \label{thm:Riccipar}
	A K\"ahler manifold is Ricci-parallel if and only if $(\nabla_{X+JX}S)(U,U)=0$, for any vector fields $X,U\in \mathfrak{X}(M)$.
\end{theorem}

\begin{proof}
	At each point $p\in M$ we define the tensor
	\begin{equation*}
		T_{p}(x,y,z):=(\nabla_{z+Jz}S)(x,y),
	\end{equation*}
	for any $x,y,z \in T_{p}M$. This tensor satisfies the following properties, derived from the properties of the Ricci tensor:
	\begin{itemize}
		\item [i)] $T_{p}(x,y,z)=T_{p}(y,x,z)$,
		\item [ii)] $T_{p}(Jx,Jy,z)=T_{p}(x,y,z)$,
	\end{itemize}
	for any $x,y,z \in T_{p}M$.
	
	On the other hand, from the linearity of the covariant derivative and the hypothesis we have that
	\begin{equation*}
		0=T_{p}(x+y,x+y,u+Ju)=2T_{p}(x,y,u+Ju),
	\end{equation*}
	hence $T_{p}(x,y,u)=-T_{p}(x,y,Ju)$. Interchanging $u$ and $Ju$ we obtain $T_{p}(x,y,Ju)=T_{p}(x,y,u)$, so necessarily $T_{p}(x,y,u)=T_{p}(x,y,Ju)=0$.
	
	Similarly, computing now $T_{p}(x+Jy,x+Jy,u+Ju)$, and again interchanging $u$ and $Ju$ in the resulting expression, we arrive at $T_{p}(x,Jy,u)=T_{p}(x,Jy,Ju)=0$.
	
	Since the above holds for any $x,y,u \in T_{p}M$, we have $\nabla S=0$.
	
	The converse is immediate.
\end{proof}

\medskip

The integrability condition of $\nabla S=0$ is $R\cdot S=0$. In the same way as for Ricci-parallel K\"ahler spaces, Ricci-semisymmetric K\"ahler spaces can be characterized in terms of holomorphic planes, as shown in the following characterization result.

\begin{theorem} \label{thm:Riccisem}
	A K\"ahler manifold $(M,g,J)$ is Ricci-semisymmetric if and only if $(R\cdot \operatorname{S})(u,u;x,Jx)=0$, for any $u,x\in T_{p}M$.
\end{theorem}

\begin{proof}
	The necessary condition is immediate. The sufficient condition follows by defining the tensor $T(x_{1},x_{2},x,y):=(R\cdot S)(x_{1},x_{2};x,y)$ and applying Proposition~\ref{prop:auxalg} to this tensor.
\end{proof}

\section{Holomorphically Ricci-pseudosymmetric K\"ahler manifolds} \label{sec:HRPS}

It is known that every K\"ahler manifold $(M^{2n},g,J)$, of real dimension $2n>4$, pseudosymmetric in the sense of Deszcz, is also Ricci-pseudosymmetric in the sense of Deszcz \cite{DDDVY}. However, the converse is not true \cite{JHSV}. Moreover, the class of Ricci-pseudosymmetric manifolds is an extension of the class of Ricci-semisymmetric manifolds. Evidently, every Ricci-semisymmetric manifold is Ricci-pseudosymmetric. However, there exist various examples of Ricci-pseudosymmetric manifolds wich are not pseudosymmetric \cite{Mut}. On the other hand, as indicated in the introduction, the tensors $R\cdot S$ and $Q(g,S)$ do not possess the same symmetries and properties with respect to the structure operator $J$. For this reason, for K\"ahler manifolds it seems reasonable to consider the alternative version of Ricci pseudosymmetry based on the complex Tachibana--Ricci tensor.

\begin{definition}
	A K\"ahler manifold $(M^{2n},g,J)$, $n\geq 2$ is said to be \textit{holomorphically Ricci-pseudosymmetric} when the tensor $R\cdot S$ satisfies
	\begin{equation}
		R\cdot S=f_{S} \hspace{0.1cm}Q^{c}(g,S),
	\end{equation}
	where $f_{S}\in \mathcal{C}^{\infty}(M)$.
\end{definition}

Although the tensors involved in both definitions of Ricci pseudosymmetry are different, it is possible to characterize holomorphic Ricci pseudosymmetry in terms of the Ricci curvature of Deszcz as defined in (\ref{RCD}), obtaining results similar to those known for the Riemannian case, \cite[Theorems 7 and 10]{JHSV}. Using Proposition~\ref{prop:auxalg}, we are in a position to give the following result.

\begin{theorem} \label{thm:HRpseud}
	A K\"ahler manifold $(M^{2n},g,J)$, $n\geq 2$, is holomorphically Ricci-pseudosymmetric if and only if for every $p\in \mathcal{U}$, every direction $d$ generated by a vector $v\in T_{p}M$, and every holomorphic plane $\pi^{h}$ contained in $T_{p}M$, such that $d$ is curvature dependent on $\pi^{h}$, the Ricci curvature of Deszcz, $L_{S}(p,d,\pi^{h})$, is independent of the direction and the plane; that is, $L_{S}(p,d,\pi^{h})=L_{S}(p)$.
\end{theorem}

\begin{proof}
	First assume that $M$ is holomorphically Ricci-pseudosymmetric, so in particular
	\begin{equation}
		R\cdot S(V,V;X,JX)=f_{S} \hspace{0.1cm}Q^{c}(g,S)(V,V;X,JX),
	\end{equation}
	for each $V,X\in \mathfrak{X}(M)$, $f_{S}\in \mathcal{C}^{\infty}(M)$ being a differentiable function.
	
	Therefore, taking into account (\ref{RH}) and the definition of Ricci curvature of Deszcz, given $p\in \mathcal{U}$ and any direction $d$, generated by a vector $v\in T_{p}M$, and a holomorphic plane $\pi^{h}=x\wedge Jx$ contained in $T_{p}M$, such that $d$ is curvature dependent on the plane $\pi^{h}$, we have that $L_{S}(p,d,\pi)=2f_{S}(p)$, so that the Ricci curvature of Deszcz is independent of the direction and the plane.
	
	Conversely, suppose now that for any point $p\in \mathcal{U}$, and any pairs consisting of a direction $d$, generated by a vector $v\in T_{p}M$, and a holomorphic plane contained in $T_{p}M$, $\pi^{h}=x\wedge Jx$, such that $d$ is curvature dependent on $\pi^{h}$, its Ricci curvature of Deszcz does not depend on such pairs; that is,
	\begin{equation}
		R\cdot S(v,v;x,Jx)=L_{S}(p) \hspace{0.1cm}Q(g,S)(v,v;x,Jx),
	\end{equation}
	which, taking into account (\ref{RH}), gives
	\begin{equation} \label{T1}
		R\cdot S(v,v;x,Jx)=f_{S}(p) \hspace{0.1cm}Q^{c}(g,S)(v,v;x,Jx),
	\end{equation}
	where $f_{S}=\dfrac{1}{2}L\in \mathcal{C}^{\infty}(\mathcal{U})$.
	
	We claim at this point that the same equality holds for any pair $(d,\pi^{h})$, not necessarily curvature dependent in $T_{p}M$, for all $p\in M$ and any differentiable extension of $f_{S}$ to $M$. Therefore, since $R\cdot S$ and $Q^{c}(g,S)$ satisfy the symmetries of Proposition~\ref{prop:auxalg}, it follows that $R\cdot S=f_{S} \hspace{0.1cm}Q^{c}(g,S)$, so that $M$ is holomorphically Ricci-pseudosymmetric.
	
	It remains to prove the claim. On one hand, in the case $p\in \mathcal{U}$, consider the differentiable function $q\in \mathcal{C}^{\infty}(T_{p}M\times T_{p}M)$ defined by $q(v,w)=Q(g,S)(v,v,w,Jw)$. It is easy to prove that the zero set of $q(v,w)$ does not contain any open subset. Fix a basis $\{e_{i},Je_{i}\}_{i=1,\ldots,n}$, and note that $q(v,w)$ is a non-identically vanishing polynomial in the components of $v$ and $w$. If the direction $d$, generated by a vector $v\in T_{p}M$, is not curvature dependent on $\pi^{h}=x\wedge Jx\subset T_{p}M$, we have that $q(v,x)=0$. It is then possible to choose a sequence of tangent vectors in $T_{p}M$, $\{v_{n}\}_{n}$ and $\{x_{n}\}_{n}$ converging to $v$ and $x$, respectively, such that $q(v_{n},x_{n})\neq 0$, $\forall n\in \mathbb{N}$. Likewise, the direction $d_{n}$, is curvature dependent on the holomorphic plane $x_{n}\wedge Jx_{n}$, for each $n\in \mathbb{N}$. Consequently, we have that
	\begin{equation}
		R\cdot S(v_{n},v_{n};x_{n},Jx_{n})=\dfrac{1}{2}L_{S}(p) \hspace{0.1cm}Q^{c}(g,S)(v_{n},v_{n};x_{n},Jx_{n}),
	\end{equation}
	for all $n\in \mathbb{N}$, so that (\ref{T1}) holds at $p\in \mathcal{U}$ by a continuity argument.
	
	On the other hand, the Tachibana--Ricci tensor vanishes identically on $\operatorname{int}(M\setminus \mathcal{U})$, since this open subset is K\"ahler--Einstein and, therefore, (\ref{T1}) is trivially satisfied for any $f_{S}$. Finally, the equality also follows on $\partial (M\setminus \mathcal{U})$ by continuity, and the claim is proved.
\end{proof}

\begin{remark}
	Note that if $(M^{2n},g,J)$ is holomorphically pseudosymmetric, there exists a differentiable function $f:M\longrightarrow \mathbb{R}$ such that $R\cdot R=f\,Q^{c}(g,R)$. In this case, $(M^{2n},g,J)$ is automatically holomorphically Ricci-pseudosymmetric, and the Ricci curvature of Deszcz $L_{S}$ is equal to the double sectional curvature $L_{R}$ \cite{AAC}.
\end{remark}

\section*{Statments and Declarations}

\subsection*{Funding}
The author declares that no funds, grants, or other support were received during the preparation of this manuscript.

\end{document}